\documentclass[12pt]{article}
\usepackage{array,amsmath}
\usepackage{amssymb}
\usepackage{graphicx,subfigure}

\makeatletter
 \usepackage{fullpage} 
 \usepackage{amsthm}

\usepackage{fullpage}

\makeatother

\begin{document}

\newtheorem{thm}{Theorem}[section]
\newtheorem{lem}[thm]{Lemma}
\newtheorem{prop}[thm]{Proposition}
\newtheorem{cor}[thm]{Corollary}
\newtheorem{con}[thm]{Conjecture}
\newtheorem{claim}[thm]{Claim}
\newtheorem{exam}[thm]{Example}
\newtheorem{defn}[thm]{Definition}
\newtheorem{cons}[thm]{Construction}
\newcommand{\di}{\displaystyle}
\def\dfc{\mathrm{def}}
\def\cF{{\cal F}}
\def\cH{{\cal H}}
\def\cT{{\cal T}}
\def\cM{{\cal M}}
\def\cA{{\cal A}}
\def\cB{{\cal B}}
\def\cG{{\cal G}}
\def\ap{\alpha'}
\def\af{\alpha'_*}
\def\df{\dfc_*}
\def\Frk{F_k^{2r+1}}
\def\nul{\varnothing} 
\def\st{\colon\,}   
\def\MAP#1#2#3{#1\colon\,#2\to#3}
\def\VEC#1#2#3{#1_{#2},\ldots,#1_{#3}}
\def\VECOP#1#2#3#4{#1_{#2}#4\cdots #4 #1_{#3}}
\def\SE#1#2#3{\sum_{#1=#2}^{#3}}  \def\SGE#1#2{\sum_{#1\ge#2}}
\def\PE#1#2#3{\prod_{#1=#2}^{#3}} \def\PGE#1#2{\prod_{#1\ge#2}}
\def\UE#1#2#3{\bigcup_{#1=#2}^{#3}}
\def\FR#1#2{\frac{#1}{#2}}
\def\FL#1{\left\lfloor{#1}\right\rfloor} 
\def\CL#1{\left\lceil{#1}\right\rceil}  
\def\L{{\lambda_1}}

\title{Spectral radius and fractional matchings in graphs}
\author{Suil O\thanks{Department of Mathematics, Simon Fraser University, Burnaby, BC, V5A 1S6, osuilo@sfu.ca. Research supported in part by an NSERC Grant of Bojan Mohar.}
}

\maketitle

\begin{abstract}
A {\it fractional matching} of a graph $G$ is a function $f$ giving each edge a
number in $[0,1]$ so that $\sum_{e \in \Gamma(v)} f(e) \le 1$ for each
$v\in V(G)$, where $\Gamma(v)$ is the set of edges incident to $v$.  The
{\it fractional matching number} of $G$, written $\af(G)$, is the maximum of
$\sum_{e \in E(G)} f(e)$ over all fractional matchings $f$. 
Let $G$ be an $n$-vertex connected graph with minimum degree $d$,
let $\lambda_1(G)$ be the largest eigenvalue of $G$,
and let $k$ be a positive integer less than $n$.
In this paper, we prove that if $\lambda_1(G) < d\sqrt{1+\frac{2k}{n-k}}$, then 
$\af(G) > \FR{n-k}{2}$. As a result, we prove $\af(G) \ge \frac{nd^2}{\lambda_1(G)^2 + d^2}$; 
we characterize when equality holds in the bound.
\end{abstract}

\section {Introduction}
Throughout this paper,  all graphs are simple with no multiple edges and no loops.
A {\it matching} of a graph $G$ is a set of disjoint edges. 
The matching number of $G$, written $\ap(G)$, is the maximum size of a matching in $G$.
The {\it adjacency matrix} $A(G)$ of $G$ is the $n$-by-$n$ matrix 
in which entry $a_{i,j}$ is 1 or 0 according to whether $v_i$ and $v_j$ are adjacent or not,
where $V(G)=\{v_1,\ldots,v_n\}$.
The {\it eigenvalues} of $G$ are the eigenvalues of its adjacency matrix $A(G)$.
Let $\lambda_1(G), \ldots,\lambda_n(G)$ be its eigenvalues in nonincreasing order.
We call $\lambda_1(G)$ the {\it spectral radius of $G$}.
In 2005, Brouwer and Haemers gave a sufficient condition on $\lambda_3(G)$ in an $n$-vertex connected $k$-regular graph $G$
for the existence of a perfect matching in $G$.
In fact, they proved that if $G$ is $k$-regular and  has no perfect matching, 
then $G$ has at least three proper induced subgraphs $H_1, H_2,$ and $H_3$, which are contained in the family $\cal H$ of graphs, 
where ${\cal H}=\{H\!: |V(H)| \text{~is odd, and~}  k|V(H)|-k+2 \le 2|E(H)| \le k|V(H)|-1\}$.
Note that if $H \in \cal H$, then $H$ is not $k$-regular but may be an induced subgraph of a $k$-regular graph. 
Also, if $k$ is even, then the upper bound on $2|E(H)|$ can be replaced by $k|V(H)|-2$, since even regular graphs cannot have a cut-edge,
where an edge $e$ in a connected graph $G$ is a {\it cut-edge} of $G$ if $G-e$ is disconnected.
By using the Interlacing Theorem~(\cite{BH,GR}, Lemma 1.6 \cite{CO}) and the fact that the spectral radius of a graph is at least its average degree, 
they proved that if $G$ is $k$-regular and has no perfect matching, then $\lambda_3(G) \ge \min_{i \in \{1,2,3\}}\lambda_1(H_i) > \min_{H \in \cal H} 2|E(H)|/|V(H)|$, where $H_i \in {\cal H}$.
The bound on $\lambda_3$ that Brouwer and Haemers found could be improved,
since equality in the bound on the spectral radius in terms of the average degree holds only when graphs are regular.
Later, Cioab{\v a}, Gregory, and Haemers~\cite{CGH}
found the minimum of $\lambda_1(H)$ over all graphs $H \in \cal H$. More generally, Cioab{\v a} and the author~\cite{CO} determined connections
between the eigenvalues of an $l$-edge-connected $k$-regular graph and its matching number when $1\le l \le k-2$.

As we mentioned above, the conditions on eigenvalues in the papers actually come from having determined the largest eigenvalue of a proper subgraph in the whole graph. Thus, the results may be improved.
In this paper, we determine connections between the spectral radius of an $n$-vertex connected graph with minimum degree $d$
and its fractional matching number. This connection comes from the whole graph and is best possible.
A {\it fractional matching} of a graph $G$ is a function $f$ giving each edge a
number in $[0,1]$ so that $\sum_{e \in \Gamma(v)} f(e) \le 1$ for each
$v\in V(G)$, where $\Gamma(v)$ is the set of edges incident to $v$.  The
{\it fractional matching number} of $G$, written $\af(G)$, is the maximum of
$\sum_{e \in E(G)} f(e)$ over all fractional matchings $f$. 
In Section 3, we prove that if $G$ is an $n$-vertex graph with minimum degree $d$ and $\af(G) \le \frac{n-k}2$,
then $\lambda_1(G) \ge  d\sqrt{1+\frac{2k}{n-k}}$. This result is best possible in the sense that there are graphs $H$ with minimum degree $d$
such that $\af(H)=\frac{n-k}2$ and $\lambda_1(H)=d\sqrt{1+\frac{2k}{n-k}}$, which is shown in Section 2.

\section {Construction}
In this section, 
we construct $n$-vertex connected graphs $H$ with minimum degree $d$ such that  $\af(H)=\frac{n-k}2$, 
and $\lambda_1(H)=d\sqrt{1+\frac{2k}{n-k}}$.

\begin{cons}\label{con} {\rm
Let $d$ and $k$  be positive integers. Let ${\cal H}(d,k)$ be the family of connected bipartite graphs $H$ with bipartition $A$ and $B$ such that: 

(i) every vertex in $A$ has degree $d$, 

(ii) $|A| = |B|+k$, and  

(iii) the degrees of vertices in $B$ are equal.}
\end{cons}

The condition (iii) of Construction~\ref{con} guarantees that the degree of vertices in $B$ is greater than $d$ because $|A| > |B|$.
Note that ${\cal H}(d,k)$ contains the complete bipartite graph $K_{d,d+k}$. In fact, there are more graphs in the family ${\cal H}(d,k)$.

\begin{exam}\label{exam} {\rm
Let $k$ be a positive integer. 
For $1 \le i \le k$, let $H_i$ be a copy of $K_{d,d+1}$ with partite sets $X_i$ and $Y_i$ of sizes $d$ and $d+1$. Delete one edge in $H_i$.
Restore the original vertex degrees by adding one edge joining $Y_i$ and $X_{i+1}$ for each $i$, with subscript taken modulo $k$.
The resulting graph is $G_k$.}
\end{exam}

Let $A=\bigcup Y_i$ and $B=\bigcup X_i$. By the construction of $G_k$, it is connected and bipartite with the partition $A$ and $B$
such that (i) every vertex in $A$ has degree $d$, (ii) $|A|=|B|+k$, and (iii) every vertex in $B$ has degree $d+1$.
Thus $G_k$ is included in ${\cal H}(d,k)$. When $k$ is even, we can also construct a graph included in ${\cal H}(d,k)$ from two copies of $K_{d,d+\frac k2}$,
so there are many graphs in ${\cal H}(d,k)$, depending on $k$.

To compute the fractional matching number for graphs in ${\cal H}(d,k)$, 
we use a fractional analogue of the famous Berge--Tutte Formula~\cite{B} for the matching number.  
The {\it deficiency} of a vertex set $S$ in a graph $G$, written $\dfc_G(S)$ or simply $\dfc(S)$,
is $o(G-S)-|S|$, where $o(K)$ is the number of components of odd order of a graph $K$.
The Berge--Tutte Formula is the equality $\alpha'(G)=\min_{S \subseteq V(G)} \FR12(n-\dfc(S))$,
where $n=|V(G)|$.  The
special case $\alpha'(G)=n/2$ reduces to Tutte's $1$-Factor Theorem~\cite{T}: a
graph $G$ has a 1-factor if and only if $o(G-S)\le|S|$ for all
$S\subseteq V(G)$.  (A {\it $1$-factor} of $G$ is a subgraph whose edges form
a perfect matching.)

For the fractional analogue, let $i(K)$ denote the number of isolated vertices
in $K$.  Let $\df(S)=i(G-S)-|S|$ and $\df(G)=\max_{S\subseteq V(G)}\df(S)$.
The fractional analogue of Tutte's $1$-Factor Theorem is that $G$ has a
fractional perfect matching if and only if $i(G-S)\le|S|$ for all $S\subseteq V(G)$
(implicit in Pulleyblank~\cite{P}), and the fractional version of the
Berge--Tutte Formula is $\af(G)=\FR12(n-\df(G))$ (see~\cite{SU}, pages 19--20).

\begin{lem}\label{lem1}
If $H\in {\cal H}(d,k)$, then $\af(H)=\frac{|V(H)|-k}2$.
\end{lem}

\begin{proof}
Since $\df(B)=|A|-|B|=k$, we have $\af(H) \le \frac{|V(H)|-k}2$. 
By taking a matching with size $|B|$,  
we have $\ap(H) \ge |B|$.
Since $\af(G) \ge \ap(G)$ for any graph $G$ and $|V(H)|=2|B|+k$, we have $$\af(H) \ge \ap(G) \ge |B| = \frac{|V(H)|-k}2,$$
which gives the desired result.
\end{proof}

To determine the spectral radius of $H$ in ${\cal H}(d,k)$, the notion of ``equitable partition" of a vertex set in a graph is used. 
Consider a partition $V(G) = V_1 \cup \cdots \cup V_s$ of the vertex
set of a graph $G$ into $s$ non-empty subsets. For $1 \le  i, j \le s$, let $b_{i,j}$ denote the average number of neighbours in $V_j$ of the vertices in $V_i$. 
The quotient matrix of this partition is the $s \times s$ matrix whose $(i, j)$-th entry equals $b_{i,j}$. The eigenvalues of the quotient matrix interlace
the eigenvalues of $G$. This partition is equitable if for each $1 \le  i, j \le s$, any vertex $v \in V_i$ has exactly $b_{i,j}$ neighbours in $V_j$. 
In this case, the eigenvalues of the quotient matrix are eigenvalues of G and the spectral radius of the quotient matrix equals the spectral radius of
G (see \cite{BH, GR} for more details).

\begin{lem}\label{lem2}
If $H \in {\cal H}(d,k)$, then $\lambda_1(H)=d\sqrt{1+\frac{2k}{|V(H)|-k}}$.
\end{lem}

\begin{proof}
Let $H$ be a graph with the partition $A$ and $B$ in ${\cal H}(d,k)$.
The quotient matrix of the partition $A$ and $B$ is

$$\begin{pmatrix}
    0 &   d \\
    \frac{d|A|}{|B|}  & 0
\end{pmatrix}.$$

The characteristic polynomial of the matrix is $x^2- d \left(\frac{d|A|}{|B|}\right) = x^2- d(d+\frac{dk}{|B|})=0$.
By the definition of ${\cal H}(d,k)$, the partition $V(H)=A\cup B$ is equitable, thus
we have $$\lambda_1(H)=\sqrt{d(d+\frac{dk}{|B|})}=d\sqrt{1+\frac k{|B|}}=d\sqrt{1+\frac{2k}{|V(H)|-k}},$$
since $|B|=\frac{|V(H)|-k}2$.
\end{proof}

\section {A relationship between $\lambda_1(G)$ and $\af(G)$}
In this section, we give a connection between the spectral radius of a graph with minimum degree $d$
and its fractional matching number. To prove Lemma~\ref{main}, we use the Perron-Frobenius Theorem.

\begin{thm} \label{GR} {\rm (See Theorem 8.8.1 in~\cite{GR})}\label{PF}
If $A$ is a real nonnegative $n\times n$ matrix, and $A_1$ is a real nonnegative $n\times n$ matrix such that $A-A_1$ is nonnegative,
then $\lambda_1(A_1) \le \lambda_1(A)$.
\end{thm}

By Theorem~\ref{GR}, if $H$ is a spanning subgraph of a graph $G$, then we have $\lambda_1(H) \le \lambda_1(G)$.
We know that the eigenvalues of $K_n$, the complete graph on $n$ vertices, are $n-1$ (with multiplicity 1) and -1 (with multiplicity $n-1$), and the eigenvalues of $P_n$, the path on $n$ vertices, are $2\cos \frac{i\pi}{n+1}$ for $i \in [n]$. Note that for some $i \in [n]-\{1\}$, the $i$-th largest eigenvalue of $P_n$ is bigger than the one of $K_n$.
Thus, even if $H$ is a spanning subgraph of a graph $G$, we cannot guarantee that $i$-th largest eigenvalue of $H$ is at most $i$-th largest eigenvalue of $G$ for all $i \in [n]$, except for $i=1$. However, if $H$ is a (vertex) induced subgraph of a graph $G$,
then for all $i \in [n]$, we have $\lambda_i(H) \le \lambda_i(G)$ by the Interlacing Theorem.

Now, we are ready to prove the main result. By the Interlacing Theorem and Theorem~\ref{GR}, we obtain an upper bound on the spectral radius
in an $n$-vertex graph $G$ with minimum degree $d$ in order to guarantee that $\ap_*(G)  > \frac{n-k}2$.

\begin{lem}\label{main}
Let $G$ be an $n$-vertex connected graph with minimum degree $d$, and
let $k$ be a real number between 0 and $n$.
If $\lambda_1(G) < d \sqrt{1+\frac{2k}{n-k}}$, then
$\af(G) > \frac {n-k}2$.
\end{lem}

\begin{proof}
Assume to the contrary that $\af \le \frac{n-k}2$.
By the fractional Berge-Tutte formula, there exists a vertex subset $S$ in $V(G)$ such that
$\af(G)=\frac 12 (n-\df(S))$, which implies $\df(S) \ge \lceil k \rceil$. Let $i(G-S)=t$, and
let $T=\{v_1,\ldots,v_t\}$ be the set of the isolated vertices in $G-S$. Note that $t \ge s+\lceil k \rceil$, where $|S|=s$.
Consider the bipartite subgraph $H$ with partition $S$ and $T$ such that $E(H)$ is the set of edges of $G$ having one end-vertex in $S$ 
and the other in $T$.  Let $a=|E(H)|$.
Since every vertex in $T$ has at least degree $d$ in $H$, the number of edges between $S$ and $T$ is at least $td$, which means that $ a \ge td$.
The quotient matrix of the partition $S$ and $T$ in $H$ is

$$\begin{pmatrix}
    0 & \frac as \\
    \frac {a}{t}  & 0
\end{pmatrix}.$$

The characteristic polynomial of the matrix is $x^2-\frac as (\frac at) = 0$. Thus by Theorem~\ref{PF} and the Interlacing Theorem, we have

$$\lambda_1(G) \ge \lambda_1(H) \ge a\sqrt{\frac 1{st}} \ge dt \sqrt{\frac1{st}} = d \sqrt{\frac ts}$$$$ \ge d\sqrt{\frac {s+\lceil k \rceil}{s}} 
= d\sqrt{1+\frac {\lceil k \rceil}s} \ge d \sqrt{1+\frac {2\lceil k \rceil}{n-k}} \ge d \sqrt{1+\frac {2 k }{n-k}}$$

since $ a\ge dt$, $t \ge s+\lceil k \rceil$, $n \ge t+s \ge 2s+k$, and $s \ge d$.
\end{proof}

Lemma~\ref{main} is best possible by Lemma~\ref{lem2} and is used to prove Theorem~\ref{main2}.

\begin{thm}\label{main2}
If $G$ is an $n$-vertex graph with minimum degree $d$, then we have 
$$\af(G) \ge \frac{nd^2}{\lambda_1(G)^2 + d^2},$$
with equality if and only if $k=\frac{n(\lambda_1(G)^2 -d^2)}{\lambda_1(G)^2+d^2}$ is an integer and $G$ is an element of ${\cal H}(d,k)$.
\end{thm}

\begin{proof}
For convenience, we denote $\lambda_1(G)$ and $\af(G)$ by $\L$ and $\af$.
Lemma~\ref{main} says that if $\L <  d \sqrt{1+\frac{2k}{n-k}}$, then $\af > \frac{n-k}2$.
Since $\frac x{n-x}$ is an increasing function of $x$ on $[0,n)$, $d \sqrt{1+\frac{2k}{n-k}}$ decreases toward $\L$ as $k$ decreases toward $z$,
where $z=\frac{n(\L^2-d^2)}{\L^2+d^2}$. Therefore for each value of $k \in (z,n)$, we have $\af > \frac{n-k}2$ by Lemma~\ref{main}.
Letting $k$ tend to $z$ and finally equal to $z$, we get $\af \ge \frac{nd^2}{\L^2+d^2}$, as claimed.

If $k=\frac{n(\L^2-d^2)}{\L^2+d^2}$  is an integer and $G \in {\cal H}(d,k)$, then by Lemma~\ref{lem1},
$\af= \frac{n-k}2=\frac{nd^2}{\L^2 + d^2}$.

For the `only if' part, assume that $\af= \frac{nd^2}{\L^2 + d^2}$. 
Equality in the bound requires equality in each step of the computation: we have $k=z$ and equalities in each inequality of the chain of inequalities 
at the end of the proof of Lemma~\ref{main}. Since $\lceil k \rceil = k$, $k$ must be an integer. Furthermore, since $a=dt$, $t=s+k$, $n=2s+k$, and $s=d$ in the proof of Lemma~\ref{main} , $G$ must be included in ${\cal H}(d,k)$.

\end{proof}

\bigskip
\centerline{\large{\bf Acknowledgement}}
The author is grateful to the referees for their valuable comments, corrections and suggestions,
which led to a great improvement of this paper.

\end{document}